\documentclass[11pt]{article}
\usepackage{graphicx} 
\usepackage{amsthm, amsmath, amsfonts, amssymb, bbm}
\usepackage{xcolor}
\usepackage{enumitem}
\usepackage{hyperref}

\newcommand{\R}{\mathbb{R}}
\newcommand{\N}{\mathbb{N}}
\newcommand{\Z}{\mathbb{Z}}
\newcommand{\E}{\mathbb{E}}
\newcommand{\Prob}{\mathbb{P}}
\newcommand{\topple}{\Psi}
\newcommand{\toppledconf}{\psi}
\newcommand{\lefteigen}{\mathsf{p}}
\newcommand{\orderp}{\mathsf{w}}
\newcommand{\statep}{\hat{q}}
\newcommand{\configurationp}{\hat{\eta}}
\newcommand{\e}{e}
\newcommand{\thresh}{\textsf{t}}
\newcommand{\inst}{\xi}

\newtheorem{example}{Example}[section]
\newtheorem{lemma}{Lemma}[section]
\newtheorem{prop}{Proposition}[section]
\newtheorem{thm}{Theorem}[section]

\usepackage[left=1in, right=1in, bottom = 1in, top = 1in]{geometry}
\def\namedlabel#1#2{\begingroup
    #2%
    \def\@currentlabel{#2}%
    \phantomsection\label{#1}\endgroup
}

\definecolor{caribbeangreen}{rgb}{0.0, 0.8, 0.6}

\title{Stabilization of stochastic networks in Markovian environment}
\author{Robin Kaiser \footnote{Technische Universität München, Germany. \texttt{ro.kaiser@tum.de}}, Martin Klötzer\footnote{Universität Innsbruck, Austria. \texttt{Martin.Kloetzer@uibk.ac.at}}, Ecaterina Sava-Huss\footnote{Universität Innsbruck, 
Austria. \texttt{Ecaterina.Sava-Huss@uibk.ac.at}} }

\begin{document}

\maketitle
\begin{abstract}
We establish criteria under which stochastic networks in a Markovian environment stabilize, thus confirming \cite[Conjecture 7.2]{levine-greco2023}. The networks evolve on finite connected graphs $G=(V,E)$, and their dynamics are encoded by  $V \times V$ toppling matrices $M$, whose columns record the expected number of topplings when the environment is in stationarity. Stabilization and non-stabilization are characterized by a parameter $\rho$ which depends on the largest eigenvalue of the matrix $M+\alpha I$, with $\alpha=1+\max\{-M(v,v):v\in V\}$. The proofs rely on the toppling random walk, in which toppled vertices are sampled according to the eigenvector associated with the largest eigenvalue of $M$.
\end{abstract}

\section{Introduction}
The term \emph{Abelian networks} was coined by Levine \cite{abelian-networks1,abelian-networks2,abelian-networks3,abelian-networks4}, but in physics Abelian networks were introduced by Dhar \cite{Dhar_sandpile1}, who called them \emph{Abelian distributed processors}. Shortly, an Abelian network on a graph $G=(V,E)$  may be seen as a system of communicating automata $(\mathcal{P}_v)_{v\in V}$ indexed by the vertices $V$ of the graph, such that each $\mathcal{P}_v$ satisfies a commutativity condition. 
In \cite{abelian-networks1}, the concept of Abelian networks is characterized via two simple axioms and it is shown that such networks obey a least-action principle, which roughly means that each processor in the network performs the minimum amount of work in order to reach a halting state. In \cite{abelian-networks2}, it is shown that under natural assumptions, finite and irreducible Abelian networks halt for all inputs. In \cite{abelian-networks3}, the critical group of an Abelian network as a generalization of the sandpile group in the Abelian sandpile model is introduced and several of its properties are investigated. The theory of networks that halt on all inputs is extended to non-halting networks in \cite{abelian-networks4}.

Examples of Abelian networks are: Abelian sandpiles \cite{abelian-sandpile-btw, abelian-sandpile-dhar}, bootstrap percolation \cite{bootstrap-hol,bootstrap-ent}, rotor networks \cite{rotor}, oil and water model \cite{oil-and-water}. Allowing the transition function in an Abelian network to depend on a probability space, one obtains a larger class of processes that are called \emph{stochastic Abelian networks} (shortly \emph{stochastic networks}), and which include classical Markov chains, branching Markov chains, internal DLA, excited walks \cite{excited-walk}, activated random walks \cite{RollaSidoravicius12}, locally Markov walks \cite{MR5020587}, and stochastic sandpiles \cite{SidoraviciusTeixeira17} as special examples. Several of these systems of interacting particles can be characterized by an Abelian property
which means that changing the order of certain interactions has no effect on the final state of the system. This property turned out to be very useful when  proving phase transitions for such systems.

In the current work, we consider the stackwise representation of stochastic networks, in which vertices of $G$ communicate with each other by sending particles along adjacent edges. The number of particles sent out is randomly  sampled, and depends on the environment at the current state, while the environment is updated according to a transition matrix. 
This work is motivated by Levine-Greco  \cite{levine-greco2023}, who asked whether the survival/extinction of a multitype branching process in a Markovian environment is determined solely by the largest eigenvalue of the expected offspring distribution matrix. We confirm \cite[Conjecture 7.2]{levine-greco2023} and give a proof that extends beyond multitype branching processes to a broader class of stochastic networks.

For stating the main result, we briefly introduce the model. Let $G=(V,E)$ be a finite connected graph. For every $v\in V$, we consider a finite set of \emph{environment states $S_v$} and a \emph{transition matrix $P_v$} that governs transitions among these states. Let $\pi_v$ be the stationary distribution of $P_v$. The \emph{expected toppling matrix $M:V\times V\to\R$} is defined as
\begin{equation}\label{eq:expected-toppling-matrix}
M(v,w):=\sum_{s\in S_v}\pi_v(s)\mu^{v,s}(w),
\end{equation}
where for $v\in V$ and $s\in S_v$ we define $\mu^{v,s}:V\to\R$ as the expectation of a given probability measure $\nu^{v,s}\in\mathsf{Prob}(\Z^V)$; 
$\nu^{v,s}$ describes how vertex $v$ interacts with the other vertices when the environment at $v$ is in state $s$.
Starting from an initial configuration $\eta_0:V\to \N_0$ and a threshold function $\thresh:V\to\N_0$, 
define a Markov chain $(\eta_n)_{n\in\N}$ as follows: given $\eta_n$, select at random a vertex $v\in V$ with $\eta_n(v)\geq \thresh(v)$, update the environment at $v$ according to $P_v$, and then add to
$\eta_n$ a column sampled from $\nu^{v,s}$, where $s\in S_v$ is the updated environment. We call $(\eta_n)_{n\in\N}$ \emph{a stochastic network in Markovian environment}. The network stabilizes if there exists $n\in\N$ such that  $\eta_n(v)\leq\thresh(v)$ for all $v\in V$. 
We introduce three assumptions, under which the main result hold. The first one \eqref{eq:only-local-reduction} states that we can reduce mass only at the toppled vertex. The second assumption \eqref{eq:bfb} means that we can reduce  mass at the toppled vertex by at most $\thresh(v)$. The last assumption \eqref{eq:irr} requires that there exists $k>0$ such that $(M-\mathsf{Diag}(M))^k>0$ on the off-diagonal entries, where $\mathrm{Diag}(M)$ denotes the matrix that coincides with $M$ on the diagonal and is zero off-diagonal.

\begin{thm}\label{thm:main}
Let $G=(V,E)$ be a finite, connected, directed graph and let $(\eta_n)_{n\in \N}$ be a stochastic network in a Markovian environment on $G$ which satisfies (\ref{eq:only-local-reduction}), (\ref{eq:bfb}) and (\ref{eq:irr}). Let $M$ be the expected offspring matrix as in \eqref{eq:expected-toppling-matrix} and 
$$\alpha=\max\{-M(v,v)\,|\,v\in V\}+1\quad \text{and} \quad
\rho=r(M+\alpha I)-\alpha,$$
 where $r(M+\alpha I)$ is the Perron-Frobenius eigenvalue of $M+\alpha I$.
\begin{enumerate}
\setlength\itemsep{0em}
\item (Subcritical case) If $\rho<0$, then for any initial state $\eta_0:V\to\N_0$, the network stabilizes almost surely.\label{item:1}
\item\label{item:2} (Critical case) If $\rho=0$, then either there exists a conserved quantity, or the network  stabilizes almost surely for every initial state.
\item\label{item:3} (Supercritical case) If $\rho>0$, then there exists an initial state $\eta_0:V\to\N_0$, for which the network does not stabilize with positive probability.
    \end{enumerate}
\end{thm}
The subcritical case is established in Proposition \ref{prop:subcritical}, the supercritical case in Proposition \ref{prop:supercritical}, and the critical case follows from Propositions \ref{prop:conserved-crit} and \ref{prop:conserved-exists}. The argument in all three cases is to sample toppled vertices according to the eigenvector $\mathsf{p}$ corresponding to the spectral radius of $M+\alpha I$, while also allowing vertices below the threshold to topple. Then stabilization under $\mathsf{p}$-sampling is equivalent to stabilization in the original dynamics in which toppling above threshold is allowed. This equivalence reduces the analysis of stabilization to standard methods in random walks theory.

\textbf{Outline.}
Section \ref{sec:prelim} introduces stochastic networks and lays out the model. In Section \ref{sec:determ-random}, we use the least-action principle to show that stabilization can be reached by exhibiting any sequence of topplings, even if vertices below the threshold are allowed to topple. Section \ref{sec:topl-rw} introduces the toppling random walk and its connection with the stochastic network. In Section \ref{sec:main-thm} we prove Theorem \ref{thm:main}.

\section{Preliminaries}\label{sec:prelim}

Let $G = (V, E)$ be a finite, connected and directed graph with vertex set $V$ and edge set $E\subset V\times V$. For each $v \in V$, fix a finite set $S_v$ of environments, and call $S := \prod_{v \in V} S_v$
\emph{the set of global environments.}

\textbf{Environment chain and toppling rules.} For each $v \in V$, we are given a stochastic matrix $P_v \in \mathbb{R}^{S_v \times S_v}$, which defines a Markovian environment at $v$, namely a Markov chain $(Y_j^v)_{j \in \mathbb{N}}$ with state space $S_v$ and transition matrix $P_v$. We assume that for each $v$ the chain $(Y_j^v)$ is irreducible and aperiodic, and denote by $\pi_v$ the stationary distribution of $P_v$, that is $\pi_vP_v=\pi_v$, for all $v\in V$. 
Fix also for all $v\in V$ and  $s\in S_v$ a probability distribution $\nu^{v,s}\in\mathsf{Prob}(\Z^V)$, where $\mathsf{Prob}(\Z^V)$ denotes the set of probability distributions on $\Z^V$, the space of functions from $V$ to $\Z$. A sample from $\nu^{v,s}$ is denoted by $\inst^{v,s}$ and its expectation is
$\mathbb{E}\big[\inst^{v,s}\big]=\mu^{v,s}\in\R^V$. Finally, we define the expected toppling matrix in stationarity $M\in\R^{V\times V}$ entrywise as in equation \eqref{eq:expected-toppling-matrix}. The matrix $M$ evaluated at $(v,w)\in V\times V$ gives the expected mass added to $w$ if $v$ has been toppled.

\textbf{Stochastic networks in Markovian environments.} Fix an initial configuration of particles $\eta_0:V\to\N_0$, an initial global environment $q_0\in S$, and a toppling threshold $\thresh:V\to\N_0$. In order to guarantee that the configuration stays non-negative throughout the process, we assume
\begin{align*}
0\leq \thresh(v)+\inst^{v,s}(v), \text{ for all } v\in V \text{ and } s\in S_v  \text{ almost surely}.
\end{align*}
Define the Markov chain $(\eta_n,q_n)_{n\in\N}$, where $\eta_n$ represents the random particle configuration at time $n$ and $q_n$ the random global environment at time $n$, which evolves as follows: given
that the state of the chain at time $n\in \N_0$ is $(\eta_n,q_n)$, we first choose a random vertex $\mathsf{v}_n$ independently of the past with probability 
\begin{equation}\label{eq2}
\mathbb{P}(\mathsf{v}_n=v)=\frac{\max\{\eta_n(v)-\thresh(v),0\}}{|\eta_n-\thresh|_+},
\end{equation}
where $|\eta_n-\thresh|_+=\sum_{v\in V}\max\{\eta_n(v)-\thresh(v),0\}$ is the total weight of the positive part of $\eta_n-\thresh$. Then, the environment at time $n+1$ is updated to
\begin{align*}
    q_{n+1}(w)=q_n(w),\text{ for }w\in V\backslash\{\mathsf{v}_n\},
\end{align*}
and for  $s\in S_{\mathsf{v}_n}$
\begin{align*}
    q_{n+1}(\mathsf{v}_n)=s, \text{ with probability }P_{\mathsf{v}_n}(q_n(\mathsf{v}_n),s).
\end{align*}
In the updated environment $s=q_{n+1}(\mathsf{v}_n)\in S_v$ we sample
a random toppling $\inst_{n+1}$ from the probability distribution $\nu^{\mathsf{v}_n,s}$ and update the particle configuration
\begin{align*}
\eta_{n+1}=\eta_n+\inst_{n+1}.
\end{align*}
The process terminates if for all $v\in V$, we have $\eta_n(v) \leq \thresh(v)$,
in which case we say that $\eta_n$ is stable or it stabilizes. We refer to both \emph{$(\eta_n)_{n\in\mathbb{N}}$ and $(\eta_n, q_n)_{n\in\mathbb{N}}$ as stochastic networks in a Markovian environment.}

In words, at each timestep one chooses a random vertex $v\in V$, whose current state is $s\in S_v$; the state of $v$ is updated to a new state $s'\in S_v$ according to the transition matrix $P_v$ of the Markov chain $(Y_j^v)_j$, and then a toppling is applied to the whole system by reducing/increasing the mass at the toppled vertices  by $\xi_{n+1}$. The process stabilizes if eventually the particle configuration drops below the given threshold $\thresh$, and it stays active if it does not stabilize.

\textbf{Stackwise representation of stochastic networks.} We also define stochastic networks using the stackwise  representation as introduced  in Diaconis-Fulton \cite{rotor}. Let $\mathcal{I}=(I^v_j)_{v\in V, j\in \N_0}$ be a stack of toppling instructions, either deterministic or random, where $I^v_j:V\to\Z$ for all $v\in V$ and $j\in\N$. We assume again that for all $v\in V$ and all $j\in \N_0$ 
$0 \leq \thresh(v)+I^v_j(v)$.
A stochastic network in stackwise representation is a Markov chain $(\eta_n,h_n)_{n\in\N}$, with initial states $\eta_0:V\to\N_0$ and $h_0:V\rightarrow\N_0$ with $h_0=0$.
If the particle configuration $\eta_n$ at time $n$ is stable for the threshold $\thresh$, then the process terminates and we set $(\eta_{n+1},h_{n+1})=(\eta_n,h_n)$. Otherwise, we choose a random vertex $\mathsf{v}_n$ with probability given in \eqref{eq2}, and define
\begin{align*}
    \eta_{n+1}=\eta_n+I^v_{h_n(\mathsf{v}_n)} \text{ and }h_{n+1}=h_n+\delta_{\mathsf{v}_n}.
\end{align*}
\textbf{Stacks sampled from the Markovian environment.} A stochastic network in Markovian environment can also be seen as a special case of a stackwise represented stochastic network in the following way. For $q_0\in S$ an initial global environment and Markov chain $(Y^v_j)_{j\in\N}$ at $v$
that starts in state $q_0(v)$, we sample the instruction $I^v_j$ according to the law of the toppling rule $\nu^{v,Y^v_j}$ independently of all other instructions. The stochastic network in a Markovian environment is then given as the stackwise represented stochastic network process with random stack $\mathcal{I}=(I^v_j)_{v\in V,j\in\N}$.

\textbf{Assumptions.}
Throughout this paper, we impose three assumptions on the toppling rules of the stochastic network that enable us to prove the main theorem. We state these assumptions in both the stochastic network frameworks introduced above.

\emph{MOLI: mass only lost internally.} First, a toppling at a vertex $v \in V$ may decrease mass only at $v$; all other vertices may only gain mass or remain unchanged: for every vertex $v \in V$, every state $s \in S_v$, and every vertex $w \neq v$,
\begin{align}\label{eq:only-local-reduction}
    \inst^{v,s}(w)\geq 0.\tag{MOLI}
\end{align}
Ths assumption holds for the stack $\mathcal{I}=(I_j^{v})_{j\in\N_0,v\in V}$, if the equation above holds for every $I_j^v$.

\emph{IRR: irreducibility.} Every pair of vertices should be able to communicate: for any $v, w \in V$, there must exist a toppling sequence that transfers mass from $v$ to $w$. That is, there exists $k \in \mathbb{N}_0$ such that for all distinct $v, w \in V$
\begin{align}\label{eq:irr}
    (M-\mathsf{Diag}(M))^k(v,w)>0,\tag{IRR}.
\end{align}
In the stackwise representation this reads as : for vertices $v\neq w$ and $n\in\N$, there exists a directed path $(v=v_1,\dots,v_k=w)$ from $v$ to $w$ such that $\Prob(I_{n+j}^{v_j}(v_{j+1})=0)<1$ for every $j\in\{1,2,\dots,k\}$.

\emph{BFB: bounded from below.}
The stochastic network should remain nonnegative during its evolution, and the mass at any $v \in V$ may be decreased by at most $\thresh(v)$. Thus for every $v \in V$ and $s \in S_v$,
\begin{align}\label{eq:bfb}
    -\inst^{v,s}(v)\leq \thresh(v).\tag{BFB}
\end{align}
Consequently, the mass reduction is uniformly bounded by $K := \max\{\thresh(v) \,|\, v \in V\}$. In the stackwise representation, this assumption holds if the same inequality is satisfied for every $I_j^v$.

We present two examples of stochastic networks in Markovian environment that fall into the framework studied in this paper.

\begin{example}[Multitype branching process in a Markovian environment]\normalfont
Consider the multitype branching process in a Markovian environment as introduced in \cite[Section 7.4]{levine-greco2023}. For a finite, connected, and directed graph $G=(V,E)$, associate to each vertex $v\in V$ a finite set of environment states $S_v$, a stochastic matrix $P_v\in\R^{S_v\times S_v}$ (called the transition matrix) and a stochastic matrix $R_v\in\R^{S_v\times \prod_{w:v\rightarrow w} \N_0}$ (called the reproduction matrix);  $v\rightarrow w$ means that $w$ is an out-neighbour of $v$. Reproduction of an individual at $v$ results in an update of the environment at $v$ according to $P_v$, followed by replacing the individual with a random number of offspring sent out to the out-neighbours of $v$. The reproduction matrix $R_v$ dictates the distribution of the offspring vector, which depends on the current environment at $v$. If no more individuals are alive, then the branching process is said to halt.
\end{example}

\begin{example}[Stochastic sandpiles]\normalfont
Let $G = (V, E)$ be a finite, connected graph, and for each $v \in V$ let $\nu^v \in \mathrm{Prob}(\mathbb{Z}^V)$ be a probability distribution such that, if $I^v \sim \nu^v$, then almost surely $I^v(v) \le \deg(v)$ and $I^v(w) \le 0$ for all $w \in V \setminus \{v\}$. A sandpile configuration is a function $\eta : V \to \mathbb{N}_0$. We call a sandpile stable if $\eta(v) \le \deg_G(v)$ for every $v \in V$. If $\eta$ is unstable at some $v \in V$, we perform a legal toppling at $v$, defined by
$T_v \eta = \eta - I^v$,
where $I^v$ is sampled from $\nu^v$, independently of everything else. We stabilize  $\eta$ by performing all possible legal topplings, and we say that $\eta$ stabilizes if only finitely many legal topplings occur during this procedure.
\end{example}

\section{Stackwise representation}\label{sec:determ-random}

\textbf{Deterministic stacks.}
We first consider the deterministic version of stackwise stochastic networks, where the toppling instruction stacks are fixed. Our first goal is to establish a least-action principle under additional assumptions on the toppling rules. As a consequence, toppling at a vertex $v$ only when the particle configuration exceeds its threshold always yields a shorter toppling sequence than one that ignores the threshold when selecting topplings.

Take an initial configuration of particles $\eta:V\rightarrow\N_0$,  a threshold function $\thresh:V\rightarrow \N_0$, and $\mathcal{I}=(I_j^v)_{v\in V, j\in\N}$ a fixed  stack of instructions with $I_j^v:V\to\Z$  for all $v\in V$ and $j\in\N$. For this set of initial data, define the toppling $\Psi_v$ at $v$ as:  for a function $h:V\rightarrow \N_0$ 
$$\Psi_v(\eta,h):=(\eta+I_{h(v)}^v,h+\delta_v),$$
where $\delta_v:V\rightarrow \Z$ is the function that is constantly $0$, except $1$ at position $v$. The toppling at $v$ is called legal if $\eta(v)>\thresh(v)$.
Informally, the operator $\Psi_v$ reads $h(v)$ -- the number of times $v$ has already toppled -- and applies the $h(v)$-th instruction at $v$ from the stack $\mathcal{I}$ to the configuration. Afterwards, it updates the counter by setting $h(v) \leftarrow h(v) + 1$ (leaving all other coordinates of $h$ unchanged).
We will use the following shorthand notation:
$$\Psi_v(\eta) := \Psi_v(\eta, 0),$$
and, for a vertex sequence $(v_1, \ldots, v_n)$ with $n \ge 2$,
$$\Psi_{(v_1,\ldots,v_n)}(\eta) := \Psi_{v_n}\bigl(\Psi_{(v_1,\ldots,v_{n-1})}(\eta)\bigr).
$$
We write $\psi_{(v_1,\ldots,v_n)}(\eta)$ for the first component of $\Psi_{(v_1,\ldots,v_n)}(\eta)$. A sequence $(v_1,\ldots,v_n)$ is called legal for $\eta$ if, for every $i \in \{1,\ldots,n\}$, the toppling at $v_i$ is legal for $\psi_{(v_1,\ldots,v_{i-1})}(\eta)$.
For a toppling sequence $(v_1,\ldots,v_n)$, the odometer function $m_{(v_1,\ldots,v_n)}$ records the number of topplings:
$$
m_{(v_1,\ldots,v_n)}(v) = \#\{\,1 \le i \le n : v_i = v\,\}.
$$
We show that the final configuration is independent of the toppling order, i.e. the model is Abelian.
\begin{lemma}[Abelian property]\label{lem:abelian-stack}
Assuming \eqref{eq:only-local-reduction}, for any initial configuration $\eta$ and any two toppling sequences $\overline{\mathsf{v}} = (v_1,\ldots,v_n)$ and $\overline{\mathsf{w}}=(w_1,...,w_m)$ with identical odometers $m_{\overline{\mathsf{v}}} = m_{\overline{\mathsf{w}}}$, we have
$
\topple_{\overline{\mathsf{v}}}(\eta) = \topple_{\overline{\mathsf{w}}}(\eta).
$
\end{lemma}
\begin{proof}
In view of the definition of the toppling operator, we have
    \begin{align*}
        \topple_{\overline{\mathsf{v}}}(\eta)=\Big(\eta+\sum_{v\in V}\sum_{i=1}^{m_{\overline{\mathsf{v}}}}I_i^v,m_{\overline{\mathsf{v}}}\Big)
        =\Big(\eta+\sum_{v\in V}\sum_{i=1}^{m_{\overline{\mathsf{w}}}}I_i^v,m_{\overline{\mathsf{w}}}\Big)
        =\topple_{\overline{\mathsf{w}}}(\eta).
    \end{align*}
\end{proof}
The next result establishes that any sequence of potentially illegal topplings must always be longer than any legal sequence of topplings.
\begin{lemma}\label{lem:moli}
Assuming \eqref{eq:only-local-reduction}, let $\eta$ be any initial particle configuration and let $\overline{\mathsf{v}} = (v_1,\ldots,v_n)$ and $\overline{\mathsf{w}} = (w_1,\ldots,w_n)$ be toppling sequences such that $m_{\overline{\mathsf{v}}}(v) = m_{\overline{\mathsf{w}}}(v)$ for some $v \in V$ and $m_{\overline{\mathsf{v}}}(w) \le m_{\overline{\mathsf{w}}}(w)$ for all $w \in V \setminus \{v\}$. Then
$
\toppledconf_{\overline{\mathsf{v}}}(\eta)(v) \le \toppledconf_{\overline{\mathsf{w}}}(\eta)(v).
$
\end{lemma}
\begin{proof}
    We have
    \begin{align*}
        \toppledconf_{\overline{\mathsf{v}}}(\eta)(v)&=\eta(v)+\sum_{i=1}^{m_{\overline{\mathsf{v}}}(v)}I_i^v(v)+\sum_{w\in V\backslash\{v\}}\sum_{i=1}^{m_{\overline{\mathsf{v}}}(w)}I_i^w(v)\\
        &\leq \eta(v)+\sum_{i=1}^{m_{\overline{\mathsf{w}}}(v)}I_i^v(v)+\sum_{w\in V\backslash\{v\}}\sum_{i=1}^{m_{\overline{\mathsf{w}}}(w)}I_i^w(v)=\toppledconf_{\overline{\mathsf{w}}}(\eta)(v),
    \end{align*}
where $m_{\overline{\mathsf{v}}}(v)=m_{\overline{\mathsf{w}}}(v)$ was used for
$\sum_{i=1}^{m_{\overline{\mathsf{v}}}(v)}I_i^v(v)=\sum_{i=1}^{m_{\overline{\mathsf{w}}}(v)}I_i^v(v)$
and \eqref{eq:only-local-reduction} together with $m_{\overline{\mathsf{v}}}(w)\leq m_{\overline{\mathsf{w}}}(w)$ for $w\in V\backslash\{v\}$ yields the upper bound
    $$\sum_{w\in V\backslash\{v\}}\sum_{i=1}^{m_{\overline{\mathsf{v}}}(w)}I_i^w(v)\leq\sum_{w\in V\backslash\{v\}}\sum_{i=1}^{m_{\overline{\mathsf{w}}}(w)}I_i^w(v).$$
\end{proof}

\begin{prop}[Least-action principle]\label{prop:least-action}
 Assume \eqref{eq:only-local-reduction} holds. For any initial configuration $\eta$ and toppling sequences $\overline{\mathsf{v}} = (v_1,\ldots,v_n)$ and $\overline{\mathsf{w}} = (w_1,\ldots,w_m)$, if $\overline{\mathsf{w}}$ is legal and $\toppledconf_{\overline{\mathsf{v}}}(\eta)$ is stable, then
$m_{\overline{\mathsf{w}}} \le m_{\overline{\mathsf{v}}}$
(pointwise).
\end{prop}
\begin{proof}
Assume the contrary, that is, there exists at least one $v\in V$ such that
$m_{\overline{\mathsf{w}}}(v)> m_{\overline{\mathsf{v}}}(v)$.
We define the sequences $\overline{\mathsf{w}}_j$, for all $j\leq m$, as
$\overline{\mathsf{w}}_j=(w_1,...,w_j)$. 
Since $m_{\overline{\mathsf{w}}}(v)> m_{\overline{\mathsf{v}}}(v)$, there must exist $j\leq m$ such that
$m_{\overline{\mathsf{w}}_j}\nleq m_{\overline{\mathsf{v}}}$.
Let $J\in\N$ be the smallest such number, and let $x\in V$ be the unique vertex such that
    $m_{\overline{\mathsf{w}}_J}(x)> m_{\overline{\mathsf{v}}}(x)$.
    We then have $m_{\overline{\mathsf{w}}_{J-1}}\leq m_{\overline{\mathsf{v}}}$ and $m_{\overline{\mathsf{w}}_{J-1}}(x)= m_{\overline{\mathsf{v}}}(x)$. Since $\overline{\mathsf{w}}$ is a legal toppling sequence, this implies $\toppledconf_{\overline{\mathsf{w}}_{J-1}}(\eta)(x)> \thresh(x)$, which together with Lemma \ref{lem:moli} yields
    $\toppledconf_{\overline{\mathsf{v}}}(\eta)(x)\geq\toppledconf_{\overline{\mathsf{w}}_{J-1}}(\eta)(x)> \thresh(x),$
and this contradicts the assumption that $\overline{\mathsf{v}}$ is stabilizing. 
\end{proof}

\textbf{Random stacks.}
We next apply Proposition \ref{prop:least-action} to show that, for stochastic networks with random stacks of toppling instructions, it suffices to find any stabilizing sequence; this ensures the existence of a legal stabilizing sequence.
Consider now the random stack $\mathcal{I}=(I^v_j)_{v\in V,j\in\N_0}$ of independent instructions distributed as
$I^v_j\sim\inst^{v,Y_j^{v}}$.
For any realization of the stack $\mathcal{I}$, we are in the deterministic setup studied previously.

\begin{prop}\label{prop:illegal-implies-legal}
Let $\mathcal{I}$ be a random instruction stack satisfying \eqref{eq:only-local-reduction}, and let $\eta$ be any initial particle configuration. If there exists a random sequence $(v_i)_{i\in \mathbb{N}}$ and some $n \in \mathbb{N}$ such that $\toppledconf_{(v_1,\ldots,v_n)}(\eta)$ is almost surely stable, then for any random sequence $(w_i)_{i\in \mathbb{N}}$ there exists $m \in \mathbb{N}$ such that $w_m$ is not a legal toppling for $\toppledconf_{(w_1,\ldots,w_{m-1})}(\eta)$.
\end{prop}

\begin{proof}
Define for the sequence $(v_i)_{i\in\N}$ the first time it stabilizes $\eta$ as
 $$\tau:=\inf\{n\in\N\,|\,\toppledconf_{(v_1,...,v_n)}(\eta)\text{ is stable}\},$$
which by assumption is finite almost surely. Assume there exists a random sequence $(w_i)_{i\in\N}$ for which all topplings are legal with positive probability. Then on the event that all topplings in the sequence $(w_i)_{i\in\N}$ are legal,  by Proposition \ref{prop:least-action} we get
$m_{(w_1,...,w_n)}\leq m_{(v_1,...,v_\tau)}$,
 for all $n\in\N$. Thus
$$n=\sum_{v\in V} m_{(w_1,...,w_n)}(v)\leq \sum_{v\in V} m_{(v_1,...,v_\tau)}(v)=\tau,$$
which implies that $\tau$ is infinite with positive probability, contradicting the  assumption.
\end{proof}

We emphasize the importance of Proposition \ref{prop:illegal-implies-legal}. First, the existence of a stabilizing toppling sequence almost surely implies that every legal toppling sequence is finite almost surely . On the other hand, the existence of an infinite legal toppling sequence rules out the existence of a legal stabilizing sequence. Therefore, Proposition \ref{prop:illegal-implies-legal} allows the proof of the Theorem \ref{thm:main} to be reduced to the construction of a toppling sequence — possibly non-legal — with the desired stabilization properties.

\section{Toppling random walk}\label{sec:topl-rw}

Recall the expected toppling matrix $M$ defined in \eqref{eq:expected-toppling-matrix}, and 
$$\alpha:=\max\{-M(v,v): v\in V\}+1.$$
We construct a random toppling sequence; we sample such a sequence according to the Perron-Frobenius eigenvector of $M+\alpha I$. For doing so, we show that this eigenvector - denoted by $\mathsf{p}$ - has all positive entries. A matrix $0\leq A\in\R^{V\times V}$ is called primitive, if there exists some $k\in\N$ such that $A^k>0$.

\begin{lemma}\label{lem:primitive}
Assuming \eqref{eq:irr} and \eqref{eq:only-local-reduction}, the matrix $M + \alpha I$ is primitive.
\end{lemma}
\begin{proof}
By the choice of $\alpha$ we have
$(M+\alpha I)(v,v)>0$,
for all $v\in V$, and for $v\neq w$ by \eqref{eq:only-local-reduction} we have 
$(M+\alpha I)(v,w)\geq  0$. The irreducibility assumption \eqref{eq:irr} yields the existence of $k\in\N$ such that $(M-\mathsf{Diag}(M))^k(v,w)>0$ for all $v\neq w$. Thus 
    $$(M+\alpha I)^k(v,w)\geq (M-\mathsf{Diag}(M))^k(v,w)>0.$$
    Furthermore, for all $v\in V$ we have that
    $$(M+\alpha I)^K(v,v)\geq (M+\alpha I)(v,v)^K\geq 1>0,$$
    and hence $(M+\alpha I)^K>0$.
\end{proof}
Thus, by Perron-Frobenius, there exists $r(M+\alpha I)>0$ (Perron-Frobenius eigenvalue of $M+\alpha I$) and $\mathsf{p}>0$ the corresponding left eigenvector normalized such that $\sum_{v\in V}\mathsf{p}(v)=1$.
Write 
$$\rho:=r(M+\alpha I)-\alpha.$$
Consider the sequence  $(\mathsf{w}_i)_{i\in\N}$, by sampling each $\mathsf{w}_i$ independently and identically distributed as
$$\mathbb{P}(\mathsf{w}_1=v)=\mathsf{p}(v),$$
for all $v\in V$. Recall that $(Y^v_j)_{j\in\N}$ is the environment chain at $v$ with initial state $q_0(v)\in S_v$, where $q_0=(q_0(v))_{v\in V}$ is the initial global environment, and the random stack is given by
$\mathcal{I}=(I^v_j)_{v\in V,j\in\N_0}$,
where $I^v_j$ has the same distribution as $\inst^{v,Y_j^{v}}$. For the initial configuration of particles $\eta_0$, write
$$\hat{\eta}_n=\toppledconf_{(\mathsf{w}_1,...,\mathsf{w}_n)}(\eta_0),$$
and define the sequence of global environments as
$$\hat{q}_n=\big(Y^v_{m_{(\mathsf{w}_1,...,\mathsf{w}_n)}(v)}\big)_{v\in V}.$$
The sequence $(\hat{\eta}_n,\hat{q}_n)_{n\in\N}$ is a Markov chain on $\Z^V\times S$, and also $(\hat{q}_n)_{n\in\N}$ and $(\mathsf{w}_{n+1},\hat{q}_n)_{n\in\N}$ are Markov chains with state spaces $S$ and $V\times S$ respectively. 
Define the sequence of stopping times $(\tau_j)_{j\in\N}$ as the consecutive return times to the initial environment $q_0$: set $\tau_0=0$ and 
for $j\geq 1$
$$\tau_j=\inf\{i>\tau_{j-1}:\hat{q}_{i}=q_0\}.$$
The sequence $(\hat{\eta}_{\tau_j},\hat{q}_{\tau_j})_{j\in\N}$ is again a Markov chain and $(\hat{\eta}_{\tau_j})_{j\in\N}$ is a random walk on $\Z^V$. Define
$$Z_j=\hat{\eta}_{\tau_j},$$
for $j\geq 0$; we call $(Z_j)_{j\in\N}$ \emph{the toppling random walk}; this is a random walk on $\Z^V$ with independent increments.

\begin{lemma}\label{lem: stationary distribution}
The stationary distribution of the Markov chain $(\mathsf{w}_{n+1},\hat{q}_n)_{n\in\N}$ is given by
    $$\Pi(v,q)=\mathsf{p}(v)\cdot\prod_{u\in V}\pi_u(q(u)).$$
\end{lemma}
\begin{proof}
For  $q\in S$ and $w\in V$, write $q^{\neq w} = (q(v))_{v\neq w}$. It is easy to see that the Markov chain $(\orderp_{n+1},\statep_n)_{n\in\N}$ has transition probabilities
$$
Q\bigl((w,r),(v,q)\bigr) = \lefteigen(v) P_w(r(w),q(w))\mathbbm{1}\{r^{\neq w} = q^{\neq w}\}
$$
where $(w,r),(v,q)\in V\times S$. 
We have
\begin{align*}
\sum_{w\in V}\sum_{r\in S}\Pi(w,r)& Q((w,r),(v,q))= \sum_{w\in V}\sum_{r\in S}\lefteigen(w) \Big(\prod_{u\in V}\pi_u(r(u))\Big) \lefteigen(v) P_w(r(w),q(w))\mathbbm{1}\{r^{\neq w} = q^{\neq w}\}\\
&=\lefteigen(v) \sum_{w\in V}\lefteigen(w)\sum_{r\in S}\Big(\pi_w(r(w))P_w(r(w),q(w))\prod_{u\neq w}\pi_u(q(u))\Big)\mathbbm{1}\{r^{\neq w} = q^{\neq w}\} \\
&= \lefteigen(v) \sum_{w\in V}\lefteigen(w)\Big(\prod_{u\neq w}\pi_u(q(u))\Big)\sum_{r\in S}\pi_w(r(w))P_w(r(w),q(w))\mathbbm{1}\{r^{\neq w} = q^{\neq w}\}\\
&= \lefteigen(v) \sum_{w\in V}\lefteigen(w)\Big(\prod_{u\neq w}\pi_u(q(u))\Big)\sum_{s\in S_w}\pi_w(s)P_w(s,q(w))\\
&= \lefteigen(v) \sum_{w\in V}\lefteigen(w)\Big(\prod_{u\neq w}\pi_u(q(u))\Big)\pi_w(q(w))\\
&= \lefteigen(v) \Big(\prod_{u\in V}\pi_u(q(u))\Big) \sum_{w\in V}\lefteigen(w) = \lefteigen(v) \prod_{u\in V}\pi_u(q(u)) = \Pi(v,q).
\end{align*}
\end{proof}

\begin{lemma}\label{lem: sum up to tau}
For all $v\in V$ and $q\in S$
$$
\E\Big[\sum_{n=0}^{\tau-1}\mathbbm{1}\{\orderp_{n+1}=v,\statep_n=q\}\Big] = \E[\tau]\Pi(v,q),
$$
where $\Pi$ is the stationary distribution from Lemma \ref{lem: stationary distribution} and $\tau$ is the first return time to the initial environment $q_0$.
\end{lemma}
\begin{proof}
By the definition of $(\tau_j)_{j\in \N_0}$ together with the Markov property, the sequence of random variables
$\big(\sum_{n=\tau_j}^{\tau_{j+1}-1}\mathbbm{1}\{\orderp_{n+1}=v,\statep_n=q\}\big)_{j\in \N_0}$
is i.i.d. with distribution $
\sum_{n=0}^{\tau-1}\mathbbm{1}\{\orderp_{n+1}=v,\statep_n=q\}$.
The expectation of each sequence term exists, because
$\sum_{n=0}^{\tau-1}\mathbbm{1}\{\orderp_{n+1}=v,\statep_n=q\} \leq \tau$,
and $\E[\tau]<\infty$. Thus, the law of large numbers and the ergodic theorem for Markov chains yields
\begin{align*}
\E\Big[\sum_{n=0}^{\tau-1}\mathbbm{1}\{\orderp_{n+1}=v,\statep_n=q\}\Big] &= \lim_{l\to\infty} \frac{1}{l}\sum_{j=1}^l \sum_{n=\tau_{j}}^{\tau_{j+1}-1}\mathbbm{1}\{\orderp_{n+1}=v,\statep_n=q\}\\
&= \lim_{l\to\infty}\frac{1}{l}\sum_{n=0}^{\tau_{l+1}-1} \mathbbm{1}\{\orderp_{n+1}=v,\statep_n=q\}
\end{align*}
which in turn equals to
\begin{align*}
&= \lim_{l\to\infty}\frac{\tau_{l+1}}{l}\frac{1}{\tau_{l+1}}\sum_{n=0}^{\tau_{l+1}-1} \mathbbm{1}\{\orderp_{n+1}=v,\statep_n=q\} \\
&=\lim_{l\to\infty}\frac{\tau_{l+1}}{l}\cdot\lim_{k\to\infty}\frac{1}{\tau_{k}}\sum_{n=0}^{\tau_{k}-1} \mathbbm{1}\{\orderp_{n+1}=v,\statep_n=q\}
= \E[\tau]\Pi(v,q).
\end{align*}
\end{proof}
Combining Lemmas \ref{lem: stationary distribution} and \ref{lem: sum up to tau} yields the expected increments of the toppling random walk $(Z_j)_{j\in\mathbb{N}}$. The next result extends \cite[Lemma 3.3]{levine-greco2023} to higher dimensions, and the two proofs are similar.
\begin{prop} \label{prop: expectation}
It holds $\E[Z_1-Z_0] = \E[\tau]\rho\lefteigen$.
\end{prop}

\begin{proof}
Without loss of generality, assume $Z_0=0$. Then
$$
Z_1 = \sum_{n=0}^{\tau-1} \sum_{v\in V} \sum_{q\in S} \mathbbm{1}\{ \orderp_{n+1} = v,\statep_n = q\}I_{m_{(\orderp_1,...,\orderp_{n})}(v)}^{v}.
$$
Taking the conditional expectation with respect to $\tau$ yields
\begin{align*}
    \E[Z_1|\tau] &= \sum_{n=0}^{\tau-1} \sum_{v\in V} \sum_{q\in S}\E\Big[\mathbbm{1}\{\orderp_{n+1} = v,\statep_n = q\}I_{m_{(\orderp_1,...,\orderp_{n})}(v)}^{v}\Big|\tau\Big]\\
        &=\sum_{n=0}^{\tau-1} \sum_{v\in V} \sum_{q\in S}\E\Big[\mathbbm{1}\{\orderp_{n+1} = v,\statep_n = q\}\Big|\tau\Big]\mu^{v,q(v)} \\
    &= \sum_{v\in V} \sum_{q\in S}\E\Big[\sum_{n=0}^{\tau-1}\mathbbm{1}\{\orderp_{n+1} = v,\statep_n = q\}\Big|\tau\Big]\mu^{v,q(v)}.
\end{align*}
Seting $S_{\neq v} = \prod_{w\neq v} S_w$ for $v\in V$, and taking again the expectation on both sides above yields
\begin{align*}
\E[Z_1] &= \sum_{v\in V} \sum_{q\in S}\E\Big[\sum_{n=0}^{\tau-1}\mathbbm{1}\{\orderp_{n+1} = v,\statep_n = q\}\Big]\mu^{v,q(v)}
= \sum_{v\in V} \sum_{q\in S} \E[\tau]\Pi(v,q)\mu^{v,q(v)} \\
&= \E[\tau] \sum_{v\in V} \sum_{q\in S} \lefteigen(v)\Big(\prod_{u\in V}\pi_u(q(u))\Big)\mu^{v,q(v)}\\
&
 = \E[\tau]\sum_{v\in V} \sum_{s\in S_v} \sum_{r\in S_{\neq v}} \lefteigen(v)\pi_v(s)\Big(\prod_{u\neq v}\pi_u(r(u))\Big)\mu^{v,s} \\
&= \E[\tau]\sum_{v\in V}\lefteigen(v)\sum_{s\in S_v}\pi_v(s)\mu^{v,s}\sum_{r\in S_{\neq v}} \Big(\prod_{u\neq v}\pi_u(r(u))\Big) \\
&= \E[\tau]\sum_{v\in V}\lefteigen(v)\sum_{s\in S_v}\pi_v(s)\mu^{v,s} = \E[\tau]\sum_{v\in V}\lefteigen(v) M(v,-) =\E[\tau]\rho \lefteigen.
\end{align*}
\end{proof}

\section{Proof of the Theorem \ref{thm:main}}\label{sec:main-thm}

Knowing the expected increments of the toppling random walk $(Z_j)_{j\in\N}$ allows to prove the stabilization/explosion of the stochastic network in Markovian environment $(\eta_n)_{n\in \N}$ in the three regimes - subcritical ($\rho<0$), supercritical ($\rho>0$), and critical ($\rho=0$).

\textbf{The subcritical regime.} This is the easiest to prove and is a direct consequence of the two previous sections.
\begin{prop}[Subcritical Case]\label{prop:subcritical}
If $\rho<0$, then for any initial configuration $\eta_0:V\to\N_0$ the sequence $(\hat{\eta}_n)_{n\in\N}$ stabilizes almost surely in finite time.
\end{prop}
\begin{proof}Since $\rho<0$,  almost surely exists some $j\in\N$ such that
$\hat{\eta}_{\tau_j}=Z_j<\thresh$,
in view of Proposition \ref{prop: expectation}, and this completes the proof.
\end{proof}

\textbf{The supercritical regime.} The supercritical case needs extra care. As in Proposition \ref{prop:supercritical}, Lemma \ref{prop: expectation} implies that, with positive probability, the toppling random walk $(Z_j)_{j\in\mathbb{N}}$ eventually crosses the threshold $\thresh$ and stays above it forever. For the particle configuration, this only guarantees that $(\hat{\eta}_j)_{j\in\mathbb{N}}$ remains above $\thresh$ at the stopping times $\tau_j$. Hence, we must additionally show that, with positive probability, the configuration does not fall below $\thresh$ at any times between successive times.

We call an initial state $(\eta, q) \in \mathbb{Z}^V \times S$ \emph{viable} if there exists $j\in\N$ such that
$$
\mathbb{P}\big(\hat{\eta}_{\tau_j} - \eta \ge 1 \ \text{and}\ \hat{\eta}_k \ge \thresh \ \text{with}\ \hat{\eta}_k \ne \thresh \ \text{for all } k \le \tau_j\big) > 0,
$$
where $\tau_j$ is the $j$-th return time of the environment to $q$. Viable configurations are central to the supercritical case; we establish that such configurations exist.
 
\begin{lemma}\label{lem:viable}
If $\rho>0$, then for any initial environment $q\in S$ there exists an initial configuration $\eta\in\N^V$ such that $(\eta,q)$ is viable.
\end{lemma}
\begin{proof}
Fix the initial environment $q\in S$, and call a sequence of environments $(q_0,...,q_l)$ an excursion, if $q_l=q_0=q$ and $q_i\neq q$ for $i\in\{1,...,l-1\}$. Define
$$
\mathsf{Cyc} := \Big\{ \big((\eta_0,\ldots,\eta_l), (q_0,\ldots,q_l)\big) :\ l \in \mathbb{N},\ \forall i \le l:\ \eta_i \in \mathbb{Z}^V,\ \eta_0 = 0,\ (q_0,\ldots,q_l)\ \text{is an excursion} \Big\}.
$$
For any $c=((\eta_0,\ldots,\eta_l),(q_0,\ldots,q_l)) \in \mathsf{Cyc}$, set $H(c)=\eta_l$, thus $H$ denotes the final particle configuration in an excursion.
Recall that $(\eta_{\tau_k})_{k\in\N}$ is a random walk with expected step size $\rho\,\mathsf{p}\,\mathbb{E}[\tau]>0$. Therefore, there exists $j\in\N$ such that
$$\mathbb{P}(\forall v\in V:\eta_{\tau_j}(v)\geq 1)>0.$$
This implies there exist excursions $c_1,...,c_j\in\mathsf{Cyc}$ with $\sum_{k=1}^j H(c_k)>1.$ Let $K\in\N$ be the constant from \eqref{eq:bfb}. Setting $c=((\eta_0,\ldots,\eta_l),(q_0,\ldots,q_l))\in \mathsf{Cyc}$  and its length $\mathsf{len}(c)=l$, then  the initial configuration $\thresh+K\sum_{k=1}^j\mathsf{len}(c_k)$ is viable for the environment $q$.
\end{proof}

\begin{prop}[Supercritical case]\label{prop:supercritical}
If $\rho > 0$, then for every environment $q \in S$ there exists an initial configuration $\eta \ge \thresh$ such that
$$
\mathbb{P}\big(\forall n \in \mathbb{N}:\ \hat{\eta}_n \ge \thresh \ \text{and}\ \hat{\eta}_n \ne \thresh\big) > 0.
$$
\end{prop}
\begin{proof}
This  is a generalization of the proof of  \cite[Theorem 3.7]{levine-greco2023}.
 For the environment $q\in S$, let $\eta\geq\thresh$ be a particle configuration such that $(\eta,q)$ is viable. Furthermore, let $j \in \mathbb{N}$ denote the number of returns to $q$ required for every component of $\eta$ to increase by $1$.
Set
$$\delta:=\mathbb{P}\big(\hat{\eta}_{\tau_j}\geq\eta+1\text{ and for all }k\leq\tau_j\text{ it holds }\hat{\eta}_k\geq \thresh\text{ and }\hat{\eta}_k\neq \thresh\big)>0.$$
Using the strong Markov property we obtain
$$\Prob\big(\configurationp_{\tau_{jn}}\geq \eta+n \text{ and for all } k\leq\tau_{jn}\text{ it holds }\configurationp_k\geq \thresh \text{ and }\configurationp_k\neq \thresh\big) \geq \delta^n > 0$$
for every $n\in\N$.
For $\mathcal{E}_n=\lbrace \configurationp_{\tau_{jn}} \geq \eta+n \rbrace$, again by the strong Markov property for $\tau_{jn}$ we get
\begin{align*}
    \Prob\big(\configurationp_k\geq \thresh & \text{ and } \configurationp_k\neq \thresh \text{ for all } k\in \N_0\big)\\
    &\geq \Prob\big(\configurationp_k\geq \thresh \text{ and } \configurationp_k\neq \thresh  \text{ for all } k>\tau_{jn}\,|\, \mathcal{E}_n,\configurationp_m\geq \thresh,\configurationp_m\neq \thresh\text{ for all }m\leq \tau_{jn}\big)\cdot\delta^n\\
    &=\Prob\big(\configurationp_k\geq \thresh \text{ and } \configurationp_k\neq\thresh \text{ for all } k>\tau_{jn}\,|\, \mathcal{E}_n\big)\cdot\delta^n.
    \end{align*}
So it suffices to find $n\in\N$ such that the right hand side above is strictly positive. By considering the complementary event, we search for an $n\in\N$ such that
$\Prob(\configurationp_k\leq \thresh \text{ for some } k\geq \tau_{jn}\,|\,\mathcal{E}_n)<1$. 
Set $\lefteigen_{\min}:=\min_{v\in V}\lefteigen(v) > 0$ and consider the following events: for $k\in\N$
\begin{align*}
 &A_k=\lbrace \configurationp_m \leq \thresh\text{ for some }m\in (\tau_k,\tau_{k+1}]\rbrace,\\
 &B_k =\Big\lbrace \configurationp_{\tau_k}\leq\thresh + \frac{\rho}{L}\lefteigen\E[\tau]k\Big\rbrace,\\
  &C_k = \Big\lbrace \tau_{k+1}-\tau_k\geq\frac{\rho\lefteigen_{\min}}{LK}\E[\tau]k\Big\rbrace,
\end{align*}
with $L\in\N$ chosen such that
$j(\rho/L)\lefteigen\E[\tau]<1$.
Observe that
$\lbrace \configurationp_k \leq \thresh \text{ for some }k>\tau_{jn}\rbrace = \bigcup_{k=jn}^\infty A_k$, and for each $k\in \N_0$ it holds
$A_k\subset B_k\cup C_k.$
This is due to \eqref{eq:bfb}; if it would  exist  $v\in V$ with
$$\configurationp_{\tau_k}(v) > \thresh(v) + \frac{\rho}{L}\lefteigen(v)\E[\tau]k,\quad \text{such that} \quad 
\tau_{k+1}-\tau_k<\frac{\rho\lefteigen_{\min}}{LK}\E[\tau]k,$$
then for every $m\in(\tau_k,\tau_{k+1}]$ it would hold
$$\configurationp_m(v) >  \thresh(v)+\frac{\rho}{L}\lefteigen(v)\E[\tau]k - K(m-\tau_k)\geq \thresh(v)+\frac{\rho}{L}\lefteigen(v)\E[\tau]k-K\frac{\rho\lefteigen_{\min}}{LK}\E[\tau]k\geq \thresh(v),$$
thus $B_k^c\cap C_k^c\subset A_k^c$.
So we need to find an $n\in\N$ such that
$\Prob\big(\bigcup_{k=jn}^\infty A_k\,\big|\, \mathcal{E}_n\big)<1$.
We first show
 \begin{equation}\label{eq: thm supercritical eq 1}
 \Prob\Bigl(\bigcup_{k=jn}^\infty C_k\,\Big|\,\mathcal{E}_n\Bigr)<\frac{1}{2}
\end{equation}
for $n$ big enough. For any $k\geq jn$ the event $C_k$ is independent of $\mathcal{E}_n$ by the strong Markov property,  hence since all $\tau_{k+1}-\tau_k$ have the same distribution as $\tau$, it holds 
 $\sum_{k=1}^\infty\Prob(C_k)<\infty$,
which together with an union bound implies the existence of $n_1\in\N$ such that  \eqref{eq: thm supercritical eq 1} holds.
Next, we show
\begin{equation}\label{eq: thm supercritical eq 2}
 \Prob\Big(\bigcup_{k=jn}^\infty B_k\,\Big|\,\mathcal{E}_n\Big)<\frac{1}{2},
 \end{equation}
 for $n$ big enough. We have
    \begin{align*}
        \Prob\Big(\exists k\geq jn:&\  \configurationp_{\tau_k}\leq \frac{\rho}{L}\mathsf{p}\E[\tau]k + \thresh\,\Big|\, \configurationp_{\tau_{jn}}\geq \eta+n\Big)\\
        &=\Prob\Big(\exists k\geq jn: (\configurationp_{\tau_k}-\configurationp_{\tau_{jn}})+\configurationp_{\tau_{jn}}\leq \frac{\rho}{L}\mathsf{p}\E[\tau]k + \thresh\,\Big|\, \configurationp_{\tau_{jn}}\geq \eta+n\Big)\\
        &\leq \Prob\Big(\exists k\geq jn: (\configurationp_{\tau_k}-\configurationp_{\tau_{jn}})+\eta + n\leq \frac{\rho}{L}\mathsf{p}\E[\tau]k+ \thresh\,\Big|\, \configurationp_{\tau_{jn}}\geq \eta+n\Big)\\
        &= \Prob\Big(\exists k\geq jn: (\configurationp_{\tau_k}-\configurationp_{\tau_{jn}})\leq \frac{\rho}{L}\mathsf{p}\E[\tau]k+ (\thresh - \eta) - n\,\Big|\, \configurationp_{\tau_{jn}}\geq \eta+n\Big)\\
        &=  \Prob\Big(\exists k\geq 0: Z_k-Z_0\leq \frac{\rho}{L}\mathsf{p}\E[\tau]k-n\Big(1-j\frac{\rho}{L}\mathsf{p}\mathbb{E}[\tau]\Big)\Big).
    \end{align*}
To show that \eqref{eq: thm supercritical eq 2} holds, it remains to upper bound the right-hand side above. Since $(Z_k)_{k}$ is a random walk on $\Z^V$, the projection on the $v$-th component is a random walk on $\Z$, which we denote by  $(R_k)_{k\in \N}$ for some fixed $v\in V$. Note that 
$\E[R_1-R_0]=\rho\mathsf{p}(v)\E[\tau]>0$.
Assume $R_0 = 0$ and set $C=1-j(\rho/L)\mathsf{p}(v)\mathbb{E}[\tau]>0$. Then
 \begin{align*}
\Prob\Big(\exists k\geq 0: Z_k\leq \frac{\rho}{L}\mathsf{p}\E[\tau]k-& n\big(1-j\frac{\rho}{L}\mathsf{p}\mathbb{E}[\tau]\big)\Big) \leq \Prob\Big(\exists k\geq 0:R_k\leq \frac{\rho}{L}\mathsf{p}(v)\E[\tau]k-Cn\Big)\\
    &\leq \sum_{k=0}^\infty \Prob\Big(R_k-\rho\mathsf{p}(v)\E[\tau]k\leq\big(\frac{\rho}{L}-\rho\big)\mathsf{p}(v)\E[\tau]k-Cn\Big)\\
    &= \sum_{k=0}^\infty \Prob\Big(-(R_k-k\E[R_1])\geq \big(\rho-\frac{\rho}{L}\big)\mathsf{p}(v)\E[\tau]k+Cn\Big).
\end{align*}
Setting $\hat{R}_k:=-R_k+k\E[R_1]$, then $(\hat{R}_k)_k$ is a centered random walk on $\Z$, and \eqref{eq:bfb} implies
$\hat{R}_1\leq \tau K+\rho\lefteigen(v)\E[\tau]$.
Since $\tau$ is the first return time of a  Markov chain on a finite state space, there exists $\Tilde{t}>0$, such that $\E\big[\e^{\Tilde{t}\tau}\big] < \infty.$
Setting $\hat{t}:=\Tilde{t}/K>0$ we get
$\E\big[\e^{\hat{t}\hat{R}_1}\big] \leq \E\big[\e^{\Tilde{t}\tau}\big]\e^{\hat{t}\rho\lefteigen(v)\E[\tau]} < \infty$.
Thus the logarithmic moment generating function
$\Lambda:[0,\hat{t}]\to \R:t\mapsto \log\E\big[\e^{t\hat{R}_1}\big]$
is continuously differentiable on $(0,\hat{t})$, fulfills $\Lambda'(0+) = \E[\hat{S}_1]=0$, and is strictly convex.
Set $c:=(\rho-\rho/L)\lefteigen(v)\E[\tau]>0$, choose an arbitrary $t\in [0,\hat{t}]$, and apply the Markov inequality to obtain
$$\Prob(\hat{R}_k\geq ck+Cn)\leq\frac{\E[\e^{t\hat{R}_k}]}{e^{t(ck+Cn)}}=\e^{-Ctn}\Big(\frac{\E[\e^{t\hat{R}_1}]}{\e^{tc}}\Big)^k = \e^{-Ctn}\e^{-(ct-\Lambda(t))k}.$$
Since $\Lambda(0)=\Lambda'(0+)=0$ and $\Lambda$ is strictly convex, it exists $t_*>0$ with  $I_*:=ct_*-\Lambda(t_*)>0$.
Therefore
\begin{align*}
\Prob\Big(\exists k\geq 0: Z_k\leq \frac{\rho}{2}\mathsf{p}\E[\tau]k-n\Big(1-j\frac{\rho}{L}\mathsf{p}\mathbb{E}[\tau]\Big)\Big)
 \leq \e^{-Ct_*n}\sum_{k=0}^\infty \e^{-I_*k}
= \e^{-Ct_*n}\frac{1}{1-\e^{-I^*}},
\end{align*}
which implies the existence of $n_2\in\N$ such that
$\mathbb{P}\big(\cup_{k=jn_2}^\infty B_k\, |\,\mathcal{E}_{n_2}\big)<\frac{1}{2}$.
For $n_3=\max\{n_1,n_2\}$
\begin{align*}
\Prob\Big(\bigcup_{k=jn_3}^\infty A_k\,\Big|\,\mathcal{E}_{n_3}\Big)\leq\Prob\Big(\bigcup_{k=jn_3}^\infty B_k\cup C_k\,\Big|\,\mathcal{E}_{n_3}\Big)
\leq\Prob\Big(\bigcup_{k=jn_3}^\infty B_k\,\Big|\,\mathcal{E}_{n_3}\Big)+\Prob\Big(\bigcup_{k=jn_3}^\infty C_k\,\Big|\,\mathcal{E}_{n_3}\Big)<\frac{1}{2}+\frac{1}{2}<1.
\end{align*}
\end{proof}

\textbf{The critical regime.}
When $\rho=0$, the total mass may be conserved during topplings, which can cause non-stabilization for sufficiently large initial configurations. To formalize this, we define conserved quantities: a function $a \in \mathbb{R}^V$ and functions $\varphi(v,\cdot): S_v \to \mathbb{R}$ constitute a conserved quantity if, almost surely for all initial configurations,
$$
\sum_{v \in V} \big(a(v)\,\eta_n(v) + \varphi(v, q_n(v))\big) = \text{const.}
$$
We first analyze the properties of $a$.

\begin{lemma}\label{lem:apos}
If $a \in \mathbb{R}^V$ and $\varphi(v,\cdot): S_v \to \mathbb{R}$ form a conserved quantity for the stochastic network $(\eta_n)_n$, then $a$ is entrywise positive; that is, $a(v) > 0$ for all $v \in V$.
\end{lemma}
\begin{proof}
Fix $w \in V$ and let $(s_0,\ldots,s_l) \in S_w^{l+1}$ be an excursion of $(Y_k^w)_{k \in \mathbb{N}}$. By choosing the initial configuration $\eta_0$ large enough at $w$, we can ensure that, with positive probability, the first $l$ topplings occur at $w$. By conservation,
$$
\sum_{v \in V} \big(a(v)\,\eta_l(v) + \varphi(v,q_l(v))\big)
= \sum_{v \in V} \big(a(v)\,\eta_0(v) + \varphi(v,q_0(v))\big).
$$
Rearranging yields $
\sum_{i=1}^l \sum_{v \in V} a(v)\,\tilde{I}_i^w(v) = 0$,
where the instructions $\tilde{I}_i^w$ are independent and $\tilde{I}_i^w \sim \inst^{w,s_i}$. Here we also used that the environment returns to its initial state. If $\tau^w$ denotes the first return time of the environment at $w$ to its initial state, then
$\sum_{i=1}^{\tau^w} \sum_{v \in V} a(v)\, I_i^w(v) = 0$,
where the instruction stack $\mathcal{I} = (I_j^v)_{v \in V,\, j \in \mathbb{N}_0}$ consists of independent instructions with $I_j^v \sim \xi^{v, Y_j^v}$. Taking conditional expectation given $\tau^w$ yields
$$
\sum_{i=1}^{\tau^w} \sum_{v \in V} a(v)\, \mu^{w, Y_i^w}(v) = 0,
$$
and by the ergodic theorem,
$$
\sum_{s \in S_w} \pi_w(s) \sum_{v \in V} a(v)\, \mu^{w,s}(v) = 0.
$$
Reordering the sums gives $M \cdot a = 0$. As in the analysis for $\mathsf{p}$, this implies $a > 0$ entrywise (here $a$ is the unique right eigenvector corresponding to $\rho = 0$).
\end{proof}
Using the entrywise positivity of $a$, we can now show that if a conserved quantity exists, then there are initial configurations that never stabilize, i.e., remain unstable indefinitely.
\begin{prop}\label{prop:conserved-crit}
If there exist a conserved quantity given by $a \in \mathbb{R}^V$ and $\varphi(v,\cdot): S_v \to \mathbb{R}$, then there is an initial state $(\eta_0, q_0) \in \mathbb{N}_0^V \times S$ such that $\eta_n$ remains unstable for all $n \in \mathbb{N}$.
\end{prop}
\begin{proof}
We choose $(\eta_0,q_0)$ such that  
    $$\sum_{v\in V}\Big(a(v)\eta_0(v)+\varphi(v,q_0(v))\Big)\geq\sum_{v\in V}a(v)\thresh(v)+|V|\cdot\max\{\varphi(v,s)\,|\,v\in V, s\in S_v\}.$$
Then for all $n\in\N$, we obtain
    \begin{align*}
        \sum_{v\in V}&a(v)\eta_n(v)=\sum_{v\in V}\Big(a(v)\eta_n(v)+\varphi(v,q_n(v))-\varphi(v,q_n(v))\Big)\\
        &\geq \sum_{v\in V}\Big(a(v)\eta_n(v)+\varphi(v,q_n(v))\Big)-|V|\cdot\max\{\varphi(v,s)\,|\,v\in V,s\in S_v\}\\
        &=\sum_{v\in V}\Big(a(v)\eta_0(v)+\varphi(v,q_0(v))\Big)-|V|\cdot\max\{\varphi(v,s)\,|\,v\in V,s\in S_v\}
        \geq\sum_{v\in V}a(v)\thresh(v),
    \end{align*}
which implies, in view of the positivity of $a$, that there exist $v\in V$ such that
$\eta_n(v)\geq \thresh(v)$. 
\end{proof}
To prove the converse — that non-stabilization entails the existence of conserved quantities — we first identify when a $d$-dimensional random walk visits the orthant $\mathcal{O} = \{x \in \mathbb{R}^d : x_i < 0 \text{ for all } i\}$.
The next lemma shows that any random walk that never enters $\mathcal{O}$ must be confined to a hyperplane in $\mathbb{R}^d$. Throughout, we use $\langle a, b \rangle = \sum_{i=1}^d a_i b_i$ for the inner product of $a,b \in \mathbb{R}^d$.

\begin{lemma}\label{lem:orthant}
Let $(X_i)_{i \in \mathbb{N}}$ be i.i.d. random vectors in $\mathbb{R}^d$ with mean zero and finite covariance matrix $\Sigma \in \mathbb{R}^{d \times d}$. Define the random walk $S_n = \sum_{i=1}^n X_i$. Then exactly one of the following holds:
\begin{enumerate}[label=(\roman*)]
\item\label{item:orthant} $\mathbb{P}(\exists\, n \in \mathbb{N} : S_n \in \mathcal{O}) = 1$, where $\mathcal{O} = \{x \in \mathbb{R}^d : x_i < 0 \text{ for all } i\}$.
\item\label{item:hyper} There exists $a \in \mathbb{R}^d$ with $a > 0$ such that $\langle X_1, a \rangle = 0$ almost surely.
\end{enumerate}
\end{lemma}

\begin{proof}
Note that if  \ref{item:hyper} holds, then $(S_n)_n$ is almost surely confined to the hyperplane $\{x \in \mathbb{R}^d : \langle x, a \rangle = 0\}$, which does not intersect $\mathcal{O}$. Thus, assume \ref{item:hyper} does not hold; we will then show that \ref{item:orthant} must hold. For $a\in\R^d$ we have
 $$\Sigma a=0\iff \mathbb{E}[X_1(i)\langle X_1,a\rangle]=0\text{ for all } i\in\{1,2,\dots,d\}.$$
Multiplying $\mathbb{E}[X_1(i)\,\langle X_1,a\rangle]$ by $a_i$ and summing over $i=1,\ldots,d$ yields
$\mathbb{E}\big[\langle X_1,a\rangle^2\big]=0$,
which implies that the walk is almost surely confined to the hyperplane orthogonal to $a$. Since \ref{item:hyper} does not hold, it follows that for every $a \in \mathbb{R}^d$ with $a > 0$, we have $\Sigma a \ne 0$. By the central limit theorem, the normalized sums $S_n/\sqrt{n}$ converge in distribution to a mean-zero Gaussian with covariance $\Sigma$. Since $\Sigma a \ne 0$ for every $a > 0$, this Gaussian has nondegenerate mass in all directions and hence assigns positive probability to the orthant $\mathcal{O}$. Hence, there exist $n_0 \in \mathbb{N}$ and a constant $C > 0$ such that for all $n \ge n_0$, it holds
$\mathbb{P}(S_n \in \mathcal{O}) > C$. Let
$\tau := \sup\{n \in \mathbb{N} : S_n \in \mathcal{O}\}$
denote the last time the random walk $S_n$ visits the orthant $\mathcal{O}$.
For any $k \in \mathbb{N}$ and $n \ge \max\{k, n_0\}$,
$$
C < \mathbb{P}(S_n \in \mathcal{O}) = \mathbb{P}(S_n \in \mathcal{O},\, \tau > k) \le \mathbb{P}(\tau > k),
$$
so $\mathbb{P}(\tau > k) > C$ for all $k \in \mathbb{N}$. Hence,
$
C < \mathbb{P}(\tau = \infty) = \mathbb{P}(S_n \in \mathcal{O} \text{ infinitely often})$.
Hewitt–Savage zero–one law implies
$$
\mathbb{P}(S_n \in \mathcal{O} \text{ infinitely often}) = 1,
$$
and therefore
$\mathbb{P}(\exists\, n \in \mathbb{N} : S_n \in \mathcal{O}) = 1.$
\end{proof}

\begin{prop}[Critical case]\label{prop:conserved-exists}
For $\rho = 0$, if there exists an initial configuration $\hat{\eta}_0 \in \mathbb{N}_0^V$ such that
$$
\mathbb{P}\big(\forall n \in \mathbb{N}\ \exists v \in V \text{ with } \hat{\eta}_n(v) \ge \thresh(v)\big) > 0,
$$
i.e., the stochastic network $(\eta_n)_{n\in \N}$ fails to stabilize with positive probability, then a conserved quantity exists for the stochastic network.
\end{prop}
\begin{proof}
Recall the toppling random walk $(Z_j)_{j \in \mathbb{N}}$ started at $\hat{\eta}_0$. Since $\rho = 0$, we have $\mathbb{E}[Z_1 - Z_0] = 0$, so $(Z_j)$ is a zero-drift random walk in $|V|$ dimensions. Moreover, by the assumption that the system does not stabilize from $\hat{\eta}_0$, the walk $(Z_j)$ avoids the set
$Q_t = \{x \in \mathbb{R}^V : x(v) < \thresh(v)\}$
with positive probability. By Lemma \ref{lem:orthant}, there exists a vector $a \in \mathbb{R}^V$ with strictly positive entries such that, almost surely for every $j \in \mathbb{N}$, it holds
$\sum_{v \in V} a(v)\, Z_j(v) = \text{const.}$
To find a conserved quantity, it remains to specify appropriate functions $\varphi(v,\cdot)$ for all $v \in V$.

As a first step, we show that the inner product with $a$ is invariant under the toppling operations. Fix $v \in V$ and $s \in S_v$, and consider a random toppling $\inst^{v,s}$ with law $\nu^{v,s}$. We couple two steps of the toppling walk so that the first toppling at $v$ is independent across the two copies, while all subsequent topplings are identical. Start the toppling walk from the zero configuration.
Pick an initial environment state $\sigma \in S$ with $\sigma(v) = s$. Let $(\sigma, \sigma_1, \ldots, \sigma_n)$ be an excursion of the environment (so $\sigma_n = \sigma$), and let $(v, v_1, \ldots, v_n)$ be the associated sequence of toppled vertices, assumed to have positive probability. Sample topplings $(I_j)_{j=1}^n$ with $I_j \sim \nu^{v_j,\, \sigma_j(v_j)}$, 
and sample $I_0$ and $I_0'$ independently from $\nu^{v,s}$. Define
$$
Z_1 = I_0 + \sum_{j=1}^n I_j,
\qquad
Z_1' = I_0' + \sum_{j=1}^n I_j.
$$
Then both $Z_1$ and $Z_1'$ are distributed as a single step of the toppling random walk and
$$\sum_{u\in V}a(u)I_0(u)+\sum_{j=1}^n\sum_{u\in V}a(u)I_j(u)=\sum_{u\in V}a(u)I_0'(u)+\sum_{j=1}^n\sum_{u\in V}a(u)I_j(u),$$
 almost surely, which implies that  $\sum_{u\in V}a(u)\inst^{v,s}(u)=\text{const.}$ almost surely. Next we construct the functions $\varphi(v,\cdot)$. Set
    $$\Phi(v,s)=\sum_{u\in V}a(u)\inst^{v,s}(u),$$
for all $v\in V$ and $s\in S_v$. 
From the first step, $\Phi(v,s)$ is constant almost surely, so the map
$\Phi: \bigcup_{v \in V} \big(\{v\} \times S_v\big) \to \mathbb{R}$
is deterministic. Moreover, for any excursion $(s_0,\ldots,s_n) \in S_v^{n+1}$ that occurs with positive probability for a given $v \in V$, we have
$\sum_{i=1}^n \Phi(v, s_i) = 0$,
since this sum corresponds to a single step of the toppling random walk.
Consequently, for any $v \in V$, any two states $s,s' \in S_v$, and any two paths $(s, s_1,\ldots,s_n, s')$ and $(s, d_1,\ldots,d_m, s')$ that each occur with positive probability,
$$
\sum_{j=1}^n \Big(\Phi(v,s_j)\Big)+\Phi(v,s')=\sum_{j=1}^m \Big(\Phi(v,d_j)\Big)+\Phi(v,s').
$$
The quantity 
$\lambda(v, s, s') = \sum_{j=1}^n \Big(\Phi(v,s_j)\Big)+\Phi(v,s')
$
is well-defined because it does not depend on the particular path chosen from $s$ to $s'$. For each $v \in V$, select a reference state $s_v \in S_v$ and set
$\varphi(v, s) = -\lambda(v, s_v, s).$
Let $\mathcal{I} = (I_j^v)_{v \in V,\, j \in \mathbb{N}_0}$ be a stack of toppling instructions. If $(v_1, \ldots, v_n) = \overline{v}$ denotes the sequence of vertices toppled in going from $\hat{\eta}_0$ to $\hat{\eta}_n$, then we have
    \begin{align*}
        \sum_{u\in V}a(u)\big(\hat{\eta}_n(u)-\hat{\eta}_0(u)\big)&=\sum_{u\in V}\sum_{w\in V}\sum_{i=1}^{\mathsf{m}_{\overline{v}}(w)}a(u)I_i^w(u)
        =\sum_{w\in V}\sum_{i=1}^{\mathsf{m}_{\overline{v}}(w)}\Phi(w,Y_i^w)\\
        &=\sum_{w\in V}\lambda(w,Y_0^w,Y_{\mathsf{m}_{\overline{v}}(w)}^w)
        =\sum_{w\in V}\Big(\lambda(w,s_w,Y_{\mathsf{m}_{\overline{v}}(w)}^w)-\lambda(w,s_w,Y_0^w)\Big)\\
        &=\sum_{w\in V}\Big(-\varphi(w,Y_{\mathsf{m}_{\overline{v}}(w)}^w)+\varphi(w,Y_0^w)\Big),
    \end{align*}
which completes the proof.
\end{proof}

\subsubsection*{Questions and remarks}
\vspace{-0.3cm}
\textbf{Infinite environments and infinite graphs $G$.}
In our model, each environment $S_v$ (for $v \in V$) is finite, and the chains $(Y_j^{v})_{j\in\mathbb{N}}$ are irreducible and aperiodic. A natural question is whether the result holds when the environments are infinite and the environment chains are positive recurrent. This should be a straightforward extension of the finite case. The same question can be asked when the graph $G=(V,E)$ is infinite as well.

\textbf{Driven dissipative system.}
Consider a subcritical stochastic network, or a critical one without a conserved quantity, on a finite graph $G=(V,E)$ with initial state $(\eta,q)$. Let the system stabilize, then pick a uniformly random vertex in $V$, add a particle, and stabilize again. The resulting sequence of stable configurations forms a Markov chain whose recurrent states are
$R \subset \mathbb{Z}^V \times S = \mathbb{Z}^V \times \prod_{v \in V} S_v$.
This chain admits a stationary distribution $\iota \in \mathrm{Prob}(R)$. What can be said about this stationary distribution?

\textbf{Infinite volume limit.}
Let $G_n = (V_n, E_n)$ be an increasing sequence of subgraphs exhausting an infinite graph $G = (V, E)$. Attach to each $v \in V$ an environment $S_v$ with chain $(Y_j^v)_{j \in \mathbb{N}}$, and for each $G_n$ consider the driven dissipative system with stationary distribution $\iota_n$. The question which arises is whether there exists a probability measure $\iota$ on $\mathbb{Z}^V \times \prod_{v \in V} S_v$ such that, for every configuration $(\eta, q)$ and every finite $A \subset V$,
 it holds $\iota_n\big((\eta, q)|_A\big) \xrightarrow[n \to \infty]{} \iota\big((\eta, q)|_A\big)$,
where $(\eta, q)|_A$ denotes the cylinder event that configurations (on $G_n$ or $G$) agree with $(\eta, q)$ on all vertices of $A$.

\textbf{Stabilization time.}
For a stochastic network in a Markovian environment on $G=(V,E)$ with initial state $(\eta,q)$, define the stabilization time
$T := \inf\{k \ge 0 : \eta_k \le \thresh\}.$
What can be said about $T$? For instance, if $T_n$ denotes the stabilization time on an increasing sequence of graphs $G_n=(V_n,E_n)$ exhausting an infinite graph $G=(V,E)$, with initial data $(\eta,q)|_{V_n}$ for some $(\eta,q) \in \mathbb{Z}^V \times \prod_{v \in V} S_v$, does there exist a scaling function $\Phi:\mathbb{N}\to\mathbb{N}$ and a constant $\gamma>0$ such that
$\lim_{n\to\infty} \frac{T_n}{\Phi(n)} = \gamma$.

\textbf{Funding information.} The research of M. Klötzer and E. Sava-Huss was funded in part by the Austrian Science Fund (FWF) 10.55776/PAT3123425. For open access purposes, the authors have applied a CC BY public copyright license to any author-accepted manuscript version arising from this submission.

\bibliographystyle{alpha}
\bibliography{lit}

\end{document}